\documentclass[11pt]{article}
\usepackage[utf8]{inputenc}

\usepackage[a4paper,left=3.2cm,top=3.2cm, right=3.2cm, bottom=3.5cm ]{geometry}
\usepackage[onehalfspacing]{setspace}

\usepackage[switch]{lineno}
\usepackage[T1]{fontenc}
\usepackage[utf8]{inputenc}

\usepackage{authblk}
\usepackage{cite}
\usepackage{tabularx}
\usepackage{hyperref}

\usepackage{amsmath,amsthm,amssymb}
\usepackage{dsfont}
\usepackage{hyperref}
\usepackage{mathtools}
\usepackage{subcaption}
\usepackage[color=green!40]{todonotes}
\usepackage{graphicx}
\usepackage{tikz}
\usepackage{bbm}
\usepackage{enumitem}
\usepackage{tikz}
\usepackage{booktabs}
\usepackage{framed}
\usepackage{algorithm}
\usepackage{algpseudocode}

\usepackage[binary-units]{siunitx}
\usepackage{eurosym}

\usepackage{todonotes}

\newcommand*{\N}{\mathds{N}}

\newcommand*{\R}{\mathds{R}}

\newcommand{\eps}{\varepsilon}

\newcommand{\reg}{\mathcal R}
\newcommand{\fun}{\mathcal F}

\newcommand{\M}{\mathcal M}
\newcommand{\gapset}{\mathcal B}

\newcommand{\Ko}{\mathbf{K}}

\newcommand{\Ao}{\mathbf A}

\newcommand{\X}{\mathbb X}
\newcommand{\Y}{\mathbb Y}

\newcommand{\tik}{\mathcal{H}}

\newcommand{\al}{\alpha}

\newcommand{\signal}{x}
\newcommand{\data}{y}
\newcommand{\datadelta}{\data^\delta}

\newcommand{\gap}{\Delta}

\newcommand{\bregman}{D}
\newcommand{\sbregman}{D^{\rm sym}}

\DeclarePairedDelimiter{\abs}{\lvert}{\rvert}
\DeclarePairedDelimiter{\norm}{\lVert}{\rVert}
\DeclarePairedDelimiter{\innerprod}{\langle}{\rangle}

\DeclarePairedDelimiter{\set}{\{}{\}}

\DeclareMathOperator{\supp}{supp}

\DeclareMathOperator{\ran}{ran}

\DeclareMathOperator*{\argmin}{arg\,min}

\DeclareGraphicsExtensions{.eps,.pdf,.png,.jpg}

\newtheorem{theorem}{Theorem}
\newtheorem{corollary}[theorem]{Corollary}

\theoremstyle{definition}
\newtheorem{proposition}[theorem]{Proposition}
\newtheorem{example}[theorem]{Example}
\newtheorem{remark}[theorem]{Remark}
\newtheorem{definition}[theorem]{Definiton}

\newtheorem{cond}{Condition}

\usepackage{soul}
\colorlet{lred}{red!40}
\colorlet{lgreen}{green!40}
\colorlet{lblue}{blue!40}
\definecolor{bananamania}{rgb}{0.98, 0.91, 0.71}

\numberwithin{equation}{section}
\numberwithin{table}{section}
\numberwithin{figure}{section}
\numberwithin{theorem}{section}

\allowdisplaybreaks
 
\title{Convergence rates for critical point regularization}

\author{Daniel Obmann}

\affil{Department of Mathematics, University of Innsbruck\authorcr
Technikerstrasse 13, 6020 Innsbruck, Austria\authorcr
E-mail:  \texttt{daniel.obmann@uibk.ac.at}
 }

\author{Markus Haltmeier}

\affil{Department of Mathematics, University of Innsbruck\authorcr
Technikerstrasse 13, 6020 Innsbruck, Austria\authorcr
E-mail:  \texttt{markus.haltmeier@uibk.ac.at}
 }

\date{February 17, 2023}

\begin{document}

\maketitle

\begin{abstract}
Tikhonov regularization involves minimizing the combination of a data discrepancy term and a regularizing term, and is the standard approach for solving inverse problems. The use of non-convex regularizers, such as those defined by trained neural networks, has been shown to be effective in many cases. However, finding global minimizers in non-convex situations can be challenging, making existing theory inapplicable. A recent development in regularization theory relaxes this requirement by providing convergence based on critical points instead of strict minimizers. This paper investigates convergence rates for the regularization with critical points using Bregman distances. Furthermore, we show that when implementing near-minimization through an iterative algorithm, a finite  number of iterations is sufficient without affecting convergence rates.
\medskip

\end{abstract}

\section{Introduction}

Many practical applications such as in medical imaging or remote sensing, can be represented by an equation of the form $    \Ao \signal + z = \datadelta$, where $\Ao \colon \X \to \Y$ describes the system, $z$ is some noise with $\norm{z} \leq \delta$, and $\signal$ is the signal of interest to be recovered. However, such  problems are often ill-posed, making direct inversion impossible or unstable.
To construct stable solutions, variational regularization methods are often used which minimize the  combination of a data-fidelity term and a regularization term. Such methods are well-established and provide stable recovery under reasonable assumptions \cite{tikhonov1963solution,groetsch1984theory,EngHanNeu96,ramlau2006tikhonov,rieder1997wavelet}. In addition, convergence rates can be derived, giving  quantitative  estimates of how close the regularized solution  is to the true solution \cite{scherzer2009variational, albani2016optimal, grasmair2010generalized, burger2004convergence,lorenz2008convergence}.

However, the theory of variational regularization assumes knowledge of exact minimizers, which can be difficult to obtain in the case of non-convex regularizers. In contrast, the critical point regularization introduced in \cite{obmann2022convergence} discards the requirements  of having access to global minimizers. Instead,  it uses critical points relative to a certain tolerance function $\phi \colon \X \to [0, \infty)$.
It has been shown that this relaxed approach leads to a stable and convergent regularization method.

In this paper, we build upon  \cite{obmann2022convergence} and derive convergence rates for critical point regularization using the absolute symmetric Bregman distance. With the additional assumption of having access to near-minimizers, we also establish convergence rates in the Bregman distance. Furthermore, we demonstrate that having access to near-minimizers is often a reasonable assumption, for instance, when an iterative minimization algorithm is available. This gives a new perspective on variational and, particularly, convex regularization methods and shows that inexactness in the minimization process for the construction of  critical points does not reduce the convergence rate.

\paragraph{Outline:}
The rest of the paper is organized as follows. Section~\ref{sec:preliminaries} gives background and an overview of the most relevant results of \cite{obmann2022convergence}. Section~\ref{sec:rates} presents   convergence rates in the absolute symmetric Bregman-distance for exact and inexact critical points. In Section~\ref{sec:gap} we consider critical points that are close to  global minimizers of the Tikhonov functional and derive additional convergence rates. Section~\ref{sec:numerics} presents a simple numerical example showing that inexact minimization may lead to significantly slower rates. The paper finishes with a short conclusion presented in   Section~\ref{sec:conclusion}.

\section{Background} \label{sec:preliminaries}

Throughout this paper, $\X$ and $\Y$ denote Hilbert spaces and $\Ao \colon \X \to \Y$ is a linear and bounded operator. We consider  the   Tikhonov functional 
\begin{equation} \label{eq:tik}
\tik_{\al, \data^\delta} = \frac{1}{2} \norm{\Ao(\cdot)- \data^\delta}^2 + \al \reg \,,
\end{equation}  
where   $\norm{\Ao \signal - \data}^2/2$ is the  data-fidelity term and $\reg$ a regularization term.

\subsection{Notation}

The following concepts are  introduced in \cite{obmann2022convergence}.

\begin{definition}[Relative sub-differentiability] \label{def:errorconvex}
Let $\reg \colon \X \to \R$  and $\phi \colon \X \to [0, \infty)$.   Then $\xi \in \X $ is called $\phi$-relative subgradient of $\reg$ at  $\signal_0 \in \X$  if
\begin{linenomath*}
\begin{equation}
 \forall \signal \in \X \colon \quad \reg(\signal_0) + \innerprod{\xi, \signal - \signal_0} \leq \reg(\signal) + \phi(\signal) \,.
\end{equation}
\end{linenomath*}
The set of all $\phi$-relative subgradients  at $\signal_0$ is denoted by  $\partial_\phi \reg(\signal_0) $ and called $\phi$-relative sub-differential of $\reg$. Functional $\reg$  is called $\phi$-relative subdifferentiable or relative subdifferentiable with bound $\phi$ if $\partial_\phi \reg(\signal)  \neq \emptyset$ for all $\signal \in \X$.
\end{definition}

\begin{definition}[Relative critical points] \label{def:criticalpoint}
 Let $\reg \colon \X \to \R$  and $\phi \colon \X \to [0, \infty)$.  We call $\signal_0 \in \X$ a $\phi$-critical point of $\reg$ or relative critical bound with bound $\phi$ if  $0 \in \partial_\phi \fun(\signal_0)$. 
\end{definition}

From the definition it follows that  $\signal_0 \in \X$ is a $\phi$-critical point if and only if  $\reg(\signal_0)  \leq \reg(\signal) + \phi(\signal)$ for all $\signal \in \X$.  

\begin{definition}[Gradient selection]
Let $\reg \colon \X \to \R$ be a relatively subdifferentiable functional. Then any function $G  \colon \X \to \X$ with $G(\signal) \in \partial_\phi \reg (\signal)$ for all $\signal \in \X$ is called gradient selection for $\reg$.
\end{definition}

Important examples of gradient selections include $G(\signal) \in \partial_0 \reg(\signal)$ if $\reg$ is convex and subdifferentiable and  $G(\signal) = \reg'(\signal)$ if $\reg$ is differentiable and  $\phi$ is such that $\reg'(\signal) \in \partial_\phi \reg(\signal)$.

\begin{definition}[Bregman-distance]
The Bregman-distance  with gradient selection $G$ of a relatively subdifferentiable $\reg$ is defined by
\begin{linenomath*}
\begin{align*}
    \bregman_G \colon \X \times \X &\to \R \colon\\
    (\signal, \signal_0) &\mapsto  \reg(\signal) -
    \reg(\signal_0) - \innerprod{G(\signal_0),\signal-  \signal_0} \,.
 \end{align*}
\end{linenomath*}
\end{definition}

If the Bregman distance is used with fixed $\signal_0$ and  $\xi = G(\signal_0)$  we write $\bregman_\xi (\signal, \signal_0)  =  \bregman_G (\signal, \signal_0) $ as  it only depends on the gradient selection at $\signal_0$.  This  is different for the symmetric Bregman distance defined next, that depends on the gradient selection  at both input elements.

\begin{definition}[Symmetric Bregman-distance]
The symmetric Bregman-distance  with gradient selection $G$ of a relatively subdifferentiable $\reg$ is defined by
\begin{linenomath*}
\begin{align*}
    \sbregman_G \colon \X \times \X &\to \R \colon\\
    (\signal, \signal_0) &\mapsto \innerprod{G(\signal) - G(\signal_0),  \signal - \signal_0} \,.
 \end{align*}
\end{linenomath*}
\end{definition}

The symmetric Bregman-distance  is actually symmetric and satisfies  $\sbregman_G(\signal, \signal_0) = \bregman_G(\signal, \signal_0) + \bregman_G(\signal_0,\signal)$. If $\reg$ is convex, then  $\bregman_G$ is non-negative with $\bregman_G(\signal, \signal_0) \leq \sbregman_G(\signal, \signal_0)$ which in particular shows that  the symmetric Bregman-distance in this case is an upper bound for the Bregman-distance. In the non-convex case, both Bregman  distances may be negative and we thus derive convergence rates for the absolute  values of it.

Throughout we write   $\al(\delta) \asymp \delta$ if $C_1 \delta \leq \al(\delta) \leq C_2 \delta$ as $\delta \to 0$ for constants $C_1, C_2 > 0$.

\subsection{Critical point regularization} 

The rates analysis use the  following assumptions on the regularization functional $\reg$.

\begin{cond}[Critical point regularization] \label{cond:conv} \hfill
\begin{enumerate}[label=(A\arabic*), leftmargin=2em, topsep=0.5em, itemsep=0em]
\item \label{cond:conv1} $\reg$ is weakly lower semicontinuous
\item \label{cond:conv2} $\reg$ is $\phi$-relatively  subdifferentiable
\item $\forall \al \,  \forall \data^\delta \colon \frac{1}{2} \norm{\Ao(\cdot)- \data^\delta}^2 + \al \reg$ is coercive. \label{cond:conv3}
\end{enumerate}
\end{cond}

Critical point regularization then consists  in finding $(\al \phi)$-critical points of the Tikhonov functional  $\tik_{\al, \data^\delta} $.  In particular, any $\signal_\al^\delta$  is such a regularized solution provided that
 \begin{linenomath*}
\begin{equation*}
 	0 \in \Ao^*(\Ao \signal_\al^\delta - \data^\delta) + \al \partial_\phi \reg(\signal_\al^\delta) \,.
\end{equation*}
\end{linenomath*}
The analysis  of  \cite{obmann2022convergence} implies the following.

\begin{theorem}[Critical point regularization]
Let $\data \in \ran(\Ao)$, $\data^\delta \in \Y$, $\al > 0$ and let Condition~\ref{cond:conv} be satisfied.
Then the following hold
\begin{enumerate}[topsep=0.5em, itemsep=0em,label= (\arabic*)]
    \item Existence: $\tik_{\al, \data^\delta}$ has a $\phi$-critical  point.
    \item Stability: Let $(\data_k)_k \in \Y^\N$ converge to $\data^\delta$ and let $(\signal_k)_k \in \X^\N$ with $0 \in \Ao^*(\Ao \signal_k - \data_k) + \al \partial_\phi \reg(\signal_k)$.
     \begin{itemize}[topsep=0.5em, itemsep=0em]
     \item  $(\signal_k)_k$ has a weakly convergent subsequence
     \item  Weak cluster points of $(\signal_k)_k$  are $(\al\phi)$-critical points of $\tik_{\al, \data^\delta}$.
     \end{itemize}
    \item Convergence: Let $(\data_k)_k \in Y^\N$ satisfy $\norm{\data - \data_k} \leq \delta_k$, let $ \delta_k, \al_k, \delta_k^2 / \al_k \to  0$ and let $(\signal_k)_k \in \X^\N$  with  $0 \in \Ao^*(\Ao \signal_k - \data_k) + \al_k \partial_\phi \reg(\signal_k)$.
    \begin{itemize}[topsep=0.5em, itemsep=0em]
    \item  $(\signal_k)_k$ has a weakly convergent subsequence and any weak cluster point  $\signal^\ddag$ of $(\signal_k)_k$ is a solution of $\Ao \signal = \data$ with  $ \reg(\signal^\ddag) \leq \inf \set{ \reg(\signal) + \phi(\signal) \mid \Ao \signal = \data }$.
      \item If the solution $\signal^\ddag$ of $\Ao \signal = \data$ is unique, then $(\signal_k)_k   \rightharpoonup  \signal^\ddag$ as $k \to \infty$.
       \item Any cluster point $\xi$ of  $\xi_k \in \partial_\phi \reg(\signal_k)$  satisfies  $\xi \in \ker(\Ao) \cap  \partial_\phi \reg(\signal^\ddag)$.
\end{itemize}
\end{enumerate}
\end{theorem}

\subsection{Example}

Before we begin our discussion of convergence  rates we start with a simple example which shows that without any further assumptions on the $\al \phi$-critical points, convergence in the value  of $\reg$ and in the Bregman-distance does not  hold.

\begin{example}[Non-convergence in $\reg$]  \label{ex:rates}
Define the operator $\Ao \colon \ell^2(\N) \to \ell^2(\N)$ by $ \Ao \signal =  (\signal_1, 0, \signal_3 / 3, \dots)$ and the regularizer $\reg  =  \norm{\cdot}_2^2/2$. Let $\data \in \ran(\Ao \Ao^*)$, $\norm{\data_k - \data} \leq \delta_k$ where $\delta_k \to 0$ and $ \supp( \data_k ) \subseteq 2\N -1$ and $\signal_k = \argmin \tik_k$ with $\tik_k \coloneqq  \tik_{\alpha_k, y_k}$, $\al_k \asymp \delta_k$.
According to standard Tikhonov regularization, $(\signal_k)_k \to  \signal^\ddag  = (\data_1, 0, 3 \data_3 , \dots)$ at the convergence rate  $\norm{\signal_k -  \signal^\ddag }^2 = \mathcal{O}(\delta_k)$. We consider two cases.

\begin{enumerate}[topsep=0.5em, itemsep=0em]
\item For given  $\varepsilon > 0$ denote $z_k = \signal_k + \eps e_{2k}$ where  $e_{2k} \in \ell^2(\N)$ has entries  $0$ except for the $(2k)$-th entry where the value is $1$. Then $\tik_k (\signal_{\eps, k})  \leq   \tik_k (\signal) +  \alpha_k \eps^2/2$  and thus $\signal_{\eps, k}$ is an $(\alpha_k \eps^2 / 2)$-critical point of $\tik_k$. Further,  $z_k \rightharpoonup \signal^\ddag$  and $\reg(z_k)  \to \reg(\signal^\ddag) + \eps^2 / 2 $ as $k \to \infty$.

\item  For $\eps_k \to 0$ consider $z_k = \signal_k + \eps_k e_{2k}$.  Then 
$z_k \rightharpoonup \signal^\ddag$  and $\abs {\reg(z_k)  - \reg(\signal^\ddag)} \leq \eps_k^2 / 2 $. In particular,  $\reg(z_k)  \to  \reg(\signal^\ddag) $ at a rate determined by $\eps_k$.
 \end{enumerate}
\end{example}

These considerations show that  without further assumptions on the critical points, the Bregman distance might  not converge to zero and even if it converges, the convergence can be arbitrarily slow. To address  this issue, we instead either work  with the symmetric Bregman-distance (see Section~\ref{sec:rates}) or   use a near-minimization concept (see Section~\ref{sec:gap}).

\section{Rates in the symmetric Bregman-distance} \label{sec:rates}

Throughout this section, let  $\reg \colon \X \to [0, \infty]$  be a possibly non-convex regularization functional that satisfies Condition~\ref{cond:conv} with tolerance function $\phi$.  Further  let $G$  be a gradient selection for $\reg$  and let the $\al\phi$-critical points   $\signal_\al^\delta$ of  $\tik_{\al, \data^\delta}$ be chosen such that $\Ao^*(\Ao \signal_\al^\delta  - \data) + \alpha G(\signal_\al^\delta) = 0$.

\subsection{Error estimates}
Convergence rates  will be derived under the following assumption on $\signal^\ddag$.

\begin{cond}[Convergence rates] \label{cond:rate} \hfill
\begin{enumerate}[label=(B\arabic*), leftmargin=2em, topsep=0.5em, itemsep=0em]
\item  \label{cond:rate1} $G(\signal^\ddag) \in \ran(\Ao^*)$
\item  \label{cond:rate2} $\exists c  \forall z \in \M(\signal^\ddag) \colon \innerprod{G(z), \signal^\ddag - z} \leq c \norm{\Ao (z -  \signal^\ddag)}$.\\
Here and below 
$\M(\signal^\ddag) = \set{\signal \in \X \colon  G(\signal) \in \ran(\Ao^*) \wedge  \abs{\reg(\signal) -  \reg(\signal^\ddag)} < \phi(\signal^\ddag) }$.\end{enumerate}
\end{cond}

We have the following result.

\begin{theorem}[Convergence rates] \label{thm:rates}
Let Condition \ref{cond:rate} hold,  let $\data \in \ran(\Ao)$,  $(\data_k)_k \in \Y^\N$ satisfy $\norm{\data_k - \data} \leq \delta_k$ where $\delta_k \to 0$ and let $\al_k \asymp \delta_k$.  Let $x_k$ satisfy $\Ao^*(\Ao \signal_k - \data_k) + \alpha_k G(\signal_k) = 0$ and let $(\signal_k)_k$ weakly converge to $\signal^\ddag$. Then the following hold
\begin{enumerate}[topsep=0.5em, itemsep=0em,label=(\arabic*) ]
    \item \label{rate1}$\norm{\Ao \signal_k - \data_k} = \mathcal{O}(\delta_k)$
    \item \label{rate2} $\abs{\sbregman_G(\signal_k, \signal^\ddag)} = \mathcal{O}(\delta_k)$.
\end{enumerate}
\end{theorem}

\begin{proof}
By definition of the symmetric Bregman distance we have
\begin{equation*}
    \abs{\sbregman_G(\signal_k, \signal^\ddag)} = \innerprod{G(\signal_k) - G(\signal^\ddag), \signal_k - \signal^\ddag} 
    + \eta_k \innerprod{G(\signal^\ddag) - G(\signal_k), \signal_k - \signal^\ddag}
 \end{equation*}
  for $\eta_k   \in \set{0, 2}$ depending on the sign of $\innerprod{G(\signal_k) - G(\signal^\ddag), \signal_k - \signal^\ddag}$.   By construction of the critical points and following the proof in \cite{obmann2022convergence} we have $\signal_k \in \M(\signal^\ddag)$ for $k$ sufficiently large.  By Condition~\ref{cond:rate},
\begin{itemize}
  \item  $ \innerprod{-G(\signal^\ddag), \signal_k - \signal^\ddag}  \leq  C_1 \left( \delta_k + \norm{\Ao \signal_k - \data_k} \right)$
  \item $  \innerprod{G(\signal^\ddag) - G(\signal_k), \signal_k - \signal^\ddag} \leq C_2 \left( \delta_k + \norm{\Ao \signal_k - \data_k} \right) $
\end{itemize}
for some constants $C_1, C_2>0$. Using the  definition of  $\signal_k$, the convexity of the data-fit term, equality $\Ao \signal^\ddag = \data$ and the estimate  $\norm{\data - \data_k} \leq \delta_k$ we get
\begin{linenomath*}
\begin{align*}
    &\frac{1}{2} \norm{\Ao \signal_k - \data_k}^2 + \alpha_k \innerprod{G(\signal_k), \signal_k - \signal^\ddag}
    \\ &= \frac{1}{2} \norm{\Ao \signal_k - \data_k}^2 + \innerprod{\Ao^* (\Ao \signal_k - \data_k), \signal^\ddag - \signal_k}
    \\& \leq \frac{1}{2} \norm{\Ao \signal^\ddag - \data_k}^2 \leq  \delta_k^2 /2 \,.
\end{align*}
\end{linenomath*}
The above  estimates together with  Young's product-inequality gives
\begin{linenomath*}
\begin{equation*}
    \frac{1}{2} \norm{\Ao \signal_k - \data_k}^2 + \alpha_k \abs{\sbregman_G(\signal_k, \signal^\ddag)}
   \leq \frac{1}{2} \delta_k^2 + C_3 \alpha_k \delta_k + C_4 \alpha_k^2 + \frac{1}{4} \norm{\Ao \signal_k - \data_k}^2 \,,
\end{equation*}
\end{linenomath*}
for  some constants $C_3, C_4>0$. With the parameter choice   $\alpha_k \asymp \delta_k$ we obtain \ref{rate1}, \ref{rate2}.
\end{proof}

\begin{corollary}[Convex case] \label{cor:convex}
If $\reg$ is convex and  $\phi =0$,  then  \ref{cond:rate1} implies  \ref{rate1}, \ref{rate2} of Theorem~\ref{thm:rates}.
\end{corollary}

\begin{proof}
Since $\reg$ is convex and $G(\signal)$ is a regular subgradient,  $\innerprod{G(z) - G(\signal^\ddag), z - \signal^\ddag} \geq 0$ for any $z \in \X$. By  \ref{cond:rate1}  we get $ \innerprod{G(z), \signal^\ddag - z} \leq \innerprod{G(\signal^\ddag), \signal^\ddag - z} \leq \norm{w} \norm{\Ao \signal^\ddag - \Ao z}$. This   implies  \ref{cond:rate2}   and thus \ref{rate1}, \ref{rate2} of Theorem~\ref{thm:rates}.
\end{proof}

Corollary~\ref{cor:convex} with $\reg(\signal) = \norm{\Ko \signal}_2^2$ for a bounded linear  $\Ko$ recovers the known rate $\norm{\Ko(\signal_k - \signal^\ddag)}^2 = \mathcal{O}(\delta_k)$.
Next we discuss regularizers which locally behave convexly around the exact solution.

\begin{remark}[Locally convex case]
Let $\data \in \ran(\Ao)$ and let $\signal^\ddag$ be a solution to $\Ao \signal = \data$ which satisfies \ref{cond:rate1}. Assume further that $\reg$ is locally convex  at $\signal^\ddag$ and that $  \reg(\signal^\ddag) + \innerprod{G(\signal^\ddag), \signal - \signal^\ddag} \leq \reg(\signal)$ for some $r > 0$ and all $\signal \in  B_r(\signal^\ddag)$.
\begin{enumerate}[topsep=0.5em,itemsep=0em, label=(\arabic*)]
\item
If $x \in B_r(\signal^\ddag)$  satsfies  $\Ao \signal = \data$, then \ref{cond:rate1}  and $\signal - \signal^\ddag \in \ker(\Ao)$ imply
\begin{linenomath*}
\begin{equation*}
    \reg(\signal^\ddag) = \reg(\signal^\ddag) + \innerprod{G(\signal^\ddag), \signal - \signal^\ddag} \leq \reg(\signal) \,.
\end{equation*}
\end{linenomath*}
Thus  $\signal^\ddag$ is a local minimizer of $\reg$ on the set of all solutions of $\Ao \signal = \data$.

\item Assume further that $\signal^\ddag$ is the weak limit of a sequence $(\signal_k)_k$ which was constructed according to Theorem~\ref{thm:rates}. Then $(\signal_k)_k$ converges to a locally $\reg$-minimizing solution of $\Ao \signal = \data$ with rates given by Theorem~\ref{thm:rates}. This recovers  a local version of the well known convergence rates for  $\reg$-minimizing solutions  \cite{scherzer2009variational}.

\item Finally, let  $\reg$ be  locally strongly convex around $\signal^\ddag$ and $  \innerprod{G(z) - G(\signal^\ddag), z - \signal^\ddag} \geq \mu \norm{z - \signal^\ddag}^2 $ for all $z \in B_r(\signal^\ddag)$ and  some $r, \mu > 0$. By Theorem~\ref{thm:rates} we get $\norm{\signal_k - \signal^\ddag} = \mathcal{O}(\sqrt{\delta_k})$, if $\signal_k \in B_r(\signal^\ddag)$. \end{enumerate}
\end{remark}

\begin{corollary}[Finite dimensional case] \label{cor:finite}
Let $\X$ be a finite dimensional,  $\reg$  coercive and $G$ bounded on bounded sets. Then \ref{cond:rate2} is satisfied and the convergence rates of Theorem~\ref{thm:rates} hold under \ref{cond:rate1}.
\end{corollary}

\begin{proof}
Since $\reg$ is coercive the set $ \M(\signal^\ddag)$ is bounded.  From  $G = \Ao^* (\Ao^*)^\ddag G$ where $(\Ao^*)^\ddag$ denotes the pseudo-inverse of $\Ao^*$ we get
\begin{linenomath*}
\begin{equation*}
    \abs{\innerprod{G(z), \signal^\ddag - z}} = \abs{\innerprod{\Ao^* (\Ao^*)^\ddag G(z), \signal^\ddag - z}}   \leq \norm{\Ao \signal^\ddag - \Ao z} \sup_{z \in \M} \norm{(\Ao^*)^\ddag G(z)}  \,.
\end{equation*}
\end{linenomath*}
Since $(\Ao^*)^\ddag$ is continuous and $G$ is bounded on bounded sets, $\sup_{z \in \M} \norm{(\Ao^*)^\ddag G(z)} < \infty$  which shows  \ref{cond:rate2}.
\end{proof}

Corollary~\ref{cor:finite} shows that  \ref{cond:rate2}  is always satisfied in the case where $\X$ is  finite dimensional. Moreover, this corollary can  readily be extended to infinite dimensional  $\X$ if  $\Ao$ has closed range.

\begin{remark}[Smooth regularizer]
Consider assumption \ref{cond:rate2}  for a smooth regularizer with $G(\signal) = \reg'(\signal)$.
Assume that $\ker(\Ao)$ is non-empty and choose $z \in \ker(\Ao)$. For $t > 0$ and $z_\pm = \signal^\ddag \pm t z$  we have  $\innerprod{G(z_\pm), \signal^\ddag - z_1} =  - t \innerprod{\reg'(\signal^\ddag \pm t z), \pm  z} \leq 0$. Adding these inequalities and dividing by $-t^2$ we find that
\begin{linenomath*}
\begin{equation*}
    \frac{1}{t} \innerprod{\reg'(\signal^\ddag + t z) - \reg'(\signal^\ddag - t z), z} \geq 0.
\end{equation*}
\end{linenomath*}
Taking $t \to 0$ yields  $D^2\reg(\signal^\ddag)(\signal_0, \signal_0) \geq 0$ which shows that $\reg$ satisfies the necessary second order convexity   around  $\signal^\ddag$ in any direction $z \in \ker(\Ao)$. Conversely, if  $z \in \X$ is such that  $\innerprod{\reg'(z), \signal^\ddag - z} > 0$, then $z - \signal^\ddag$ is a first order descent direction of $\reg$ and  walking away from $\signal^\ddag$ in the direction $z - \signal^\ddag$ has to result in an appropriate increase in data-error. \end{remark}

\subsection{Converse result}

Next we show that the source  condition is not only sufficient, but essentially also necessary for the convergence rates to hold.

\begin{proposition}[Necessity of source condition] \label{prop:necessity}
Let $\data, (\data_k)_k, (\signal_k)_k$, $(\delta_k)_k$, $(\al_k)_k$, $\signal^\ddag$ be as in Theorem~\ref{thm:rates} and assume  that $G$ is weakly continuous. Then  $ \norm{\Ao \signal_k - \data_k} = \mathcal{O}(\delta_k)$ implies the range condition \ref{cond:rate1} and the convergence rate $\sbregman_G(\signal_k, \signal^\ddag) = \mathcal{O}(\delta_k)$.
\end{proposition}

\begin{proof}
From  $\norm{\Ao \signal_k - \data_k} = \mathcal{O}(\delta_k)$ it follows  that  $w_k = (\Ao \signal_k - \data_k) / \al_k$ is bounded. Hence it has subsequence, again denoted  by $(w_k)_k$, which weakly converges to $w$.
Because $\Ao$ is continuous, $G(\signal_k)  = -\Ao^* w_k \rightharpoonup -\Ao^* w$.  Due to the weak continuity of $G$ and the uniqueness of the limit it follows that $G(\signal^\ddag) =  - \Ao^* w$. Moreover,
\begin{linenomath*}
\begin{equation*}
    \abs{\sbregman_G(\signal_k, \signal^\ddag)}
    = \abs{\innerprod{\Ao^*(w_k - w), \signal_k - \signal^\ddag}} 
    \leq C \left( \delta_k + \norm{\Ao \signal_k - \data_k} \right)
\end{equation*}
\end{linenomath*}
implies the desired rate.
\end{proof}

Proposition~\ref{prop:necessity} suggests the Morozovs discrepancy principle (see \cite{bonesky2008morozov, anzengruber2009morozov} and references therein)  for selecting the regularization parameter, since in this case the condition $\norm{\Ao \signal_k - \data_k} = \mathcal{O}(\delta_k)$ is satisfied by definition. While  this  is an interesting line of future research such an analysis is beyond the scope of this paper.

\subsection{Inexact critical points}

In the following  we consider  inexact critical points where   $\norm{\Ao^*(\Ao \signal_k - \data_k) + \al_k G(\signal_k)}$ is sufficiently small. This case has also been analyzed in \cite{obmann2022convergence} where it  has been shown to yield to a convergent regularization. The following theorem provides rates in this case.

\begin{theorem}[Inexact  rates] \label{thm:ratesinexact}
Let Condition~\ref{cond:rate} hold, $\data \in \ran(\Ao)$,  $(\data_k)_k \in \Y^\N$,  $\norm{\data_k - \data} \leq \delta_k  \to 0$ and $\al_k  \asymp \delta_k$.
Assume that $(\signal_k)_k$ is such that for $z_k = \Ao^*(\Ao \signal_k - \data_k) + \al_k G(\signal_k)$ we have  $\norm{z_k} \leq \al_k \eta_k$ for some $\eta_k \to 0$ and $\innerprod{z_k, \signal_k} \leq 0$.
Denote by $\signal^\ddag$ the weak limit of $(\signal_k)_k$.
Then the following hold
\begin{enumerate}[topsep=0.5em, itemsep=0em,label=(\arabic*)]
    \item \label{inexact-rate1} $\norm{\Ao \signal_k - \data_k}  = \mathcal{O}\Bigl(\sqrt{\delta_k^2 + \delta_k \eta_k} \Bigr)$
    \item \label{inexact-rate2} $\abs{\sbregman_G(\signal_k, \signal^\ddag)} =  \mathcal{O}(\delta_k + \eta_k)$.
\end{enumerate}
\end{theorem}

 \begin{proof}
The proof is similar  to the one  of Theorem~\ref{thm:rates} and we only show main changes. By construction 
\begin{linenomath*}
\begin{align*}
    \alpha_k & \innerprod{G(\signal_k), \signal_k - \signal^\ddag}  \\
    &= \innerprod{z_k, \signal_k - \signal^\ddag} + \innerprod{\Ao^* (\Ao \signal_k - \data_k), \signal^\ddag - \signal_k}\\
    &\leq \norm{z_k} \norm{\signal^\ddag} + \innerprod{\Ao^* (\Ao \signal_k - \data_k), \signal^\ddag - \signal_k}\\
    &\leq C\delta_k \eta_k + \innerprod{\Ao^* (\Ao \signal_k - \data_k), \signal^\ddag - \signal_k} \,.
\end{align*}
\end{linenomath*}
Hence,
\begin{linenomath*}
\begin{equation*}
    \frac{1}{2} \norm{\Ao \signal_k - \data_k}^2 + \alpha_k \innerprod{G(\signal_k), \signal_k - \signal^\ddag} \leq \frac{\delta_k^2}{2}  + C\delta_k \eta_k \,.
\end{equation*}
\end{linenomath*}
Following the proof of Theorem~\ref{thm:rates} yields \ref{inexact-rate1}, \ref{inexact-rate2}.
\end{proof}

Theorem~\ref{thm:ratesinexact} provides convergence rates dependent  on $\eta_k$ as a measure for the exactness of the critical points. To recover the rates of Theorem~\ref{thm:rates}  the choice $\eta_k = \delta_k$ is appropriate. While Theorem~\ref{thm:ratesinexact} requires $\eta_k \to 0$ the same proof  can be given for  a bounded sequence $(\eta_k)_k$. In this case the absolute symmetric Bregman-distance might not converge and the convergence in the discrepancy is only $\mathcal{O}\bigl( \sqrt{\delta_k} \bigr)$.

\begin{remark}[Iterative minimization]
Suppose  that the critical points are approximated using an iterative descent algorithm and write $z_k = \Ao^*(\Ao \signal_k - \data_k) + \al_k G(\signal_k)$ where $x_k$ is the approximate critical point. In this context the conditions of Theorem~\ref{thm:ratesinexact} on $z_k$ are  quite natural. First, $\innerprod{z_k, x_k} <  0 $ is a descent  condition for  a descent direction. Second, the condition $\norm{z_k} \leq \al_k \eta_k$ is simply a common stopping criterion for the iteration  dictated by the noise-level and the desired accuracy.
\end{remark}

Note that the necessity of the source condition for inexact critical points can be derived as in Proposition~\ref{prop:necessity}. In Section~\ref{sec:numerics} we will provide an example showing that the inexact choice can indeed result in a significantly slower convergence rate.

\section{Rates for near minimizers} \label{sec:gap}

In the previous section we have derived convergence rates for exact and inexact critical points which can be very different from global minimizers.  In this section we derive convergence rates for near-minimizers  in the absolute Bregman-distance extending results of  \cite{li2020nett}.

Throughout this section let  $\data \in \ran(\Ao)$ and $(\data_k)_k \in \Y^\N$ be a sequence of noisy data with $\norm{\data_k - \data} \leq \delta_k$. Further assume $\delta_k \to 0$ and  $\al_k  \asymp \delta_k$. Let $\signal_k$ be an $(\al_k\phi)$-critical point of the Tikhonov functional $\tik_k = \frac{1}{2} \norm{\Ao  (\cdot) - \data_k}^2 + \al_k \reg$ where the regularizer $\reg \colon \X \to [0, \infty)$ satisfies  Condition~\ref{cond:conv}. Further, let $\signal^\ddag$ be the  weak limit $(\signal_k)_{k \in \N}$.

\subsection{Error estimates} 

We call $\gap_k(\signal) = \tik_k(x_k) - \tik_k(\signal)$ the Tikhonov gap between $\signal_k$ and $\signal \in \X$. Convergence rates will then be derived under the following assumption.

\begin{cond}[Rates using Tikhonov gap] \label{cond:rates-gap} \hfill
\begin{enumerate}[label=(C\arabic*), leftmargin=2em, topsep=0.5em, itemsep=0em]
\item  \label{cond:gap1} $ \exists \xi \in \partial_\phi \reg(\signal^\ddag) \cap \ran(\Ao^*)$.
\item  \label{cond:gap2}  $\exists c \,  \forall z \in \gapset(\signal^\ddag) \colon
    \reg(\signal^\ddag) - \reg(z) \leq c \norm{\Ao z - \Ao \signal^\ddag}$.
Here   $\gapset(\signal^\ddag) \coloneqq \set{\signal \in \X \colon \abs{\reg(\signal^\ddag) - \reg(\signal)} \leq \eps}$ where $\eps >   \max \set{0, \sup_k \gap_k / \al_k}$. 
  \end{enumerate}
\end{cond}

We have the following result.

\begin{theorem}[Rates using Tikhonov gap] \label{thm:ratesgap}
Let $\data, \data_k, \delta_k,\al_k$ and $\signal^\ddag$ be as introduced above and let  Assumption \ref{cond:rates-gap} be satisfied. Then 
\begin{enumerate}[label =(\arabic*), topsep=0.5em, itemsep=0em]
    \item\label{thm:ratesgap1} $\norm{\Ao \signal_k - \data_k} = \mathcal{O}\bigl(\sqrt{\delta_k^2 + \gap_k} \bigr)$
    \item\label{thm:ratesgap2} $\abs{\bregman_\xi(\signal_k, \signal^\ddag)} = \mathcal{O}(\delta_k + \gap_k / \delta_k)$.
\end{enumerate}
\end{theorem}

\begin{proof}
By definition of $\bregman_\xi$,
\begin{linenomath*}
\begin{multline*}
    \abs{\bregman_\xi(\signal_k, \signal^\ddag)} = \reg(\signal_k) - \reg(\signal^\ddag) - \innerprod{r, \signal_k - \signal^\ddag}  \\
    + \eta_k \left(\reg(\signal^\ddag) - \reg(\signal_k) + \innerprod{r, \signal_k - \signal^\ddag} \right),
\end{multline*}
\end{linenomath*}
with $\eta_k \in \set{0, 2}$. By \ref{cond:gap1}, \ref{cond:gap2} 
\begin{linenomath*}
\begin{equation*}
    \reg(\signal^\ddag) - \reg(\signal_k) + \innerprod{r, \signal_k - \signal^\ddag} \leq C_1 \norm{\Ao \signal_k - \Ao \signal^\ddag} \leq C_1 \left( \norm{\Ao \signal_k - \data_k} + \delta_k \right),
\end{equation*}
\end{linenomath*}
By definition of the Tikhonov gap $\gap_k$, 
\begin{linenomath*}
\begin{equation*}
    \frac{1}{2} \norm{\Ao \signal_k - \data_k}^2 + \al_k \reg(\signal_k) - \al_k \reg(\signal^\ddag)
    \leq \frac{1}{2} \norm{\Ao \signal^\ddag - \data_k}^2 + \gap_k \leq \delta_k^2/2 + \gap_k.
\end{equation*}
\end{linenomath*}
Using once again \ref{cond:gap1} it follows that
\begin{linenomath*}
\begin{align*}
    &\frac{1}{2} \norm{\Ao \signal_k - \data_k}^2 + \al_k \abs{\bregman_\xi(\signal_k, \signal^\ddag)}
    \\ &\leq
    \gap_k + \delta_k^2/2 + \Tilde{C} \al_k \left( \norm{\Ao \signal_k - \data_k} + \delta_k \right) \\
    &\leq \gap_k +  \delta_k^2/2 + C_2 (\al_k \delta_k + \al_k^2) +  \norm{\Ao \signal_k - \data_k}^2/4\,,
\end{align*}
\end{linenomath*}
where the last inequality follows again by Young's product-inequality. With $\al_k \asymp \delta_k$ we get \ref{thm:ratesgap1}, \ref{thm:ratesgap2}.
\end{proof}

Clearly, if $\gap_k \leq \delta_k^2$ we obtain the classical convergence rates. Note, however, that $\gap_k$ depends on $\signal^\ddag$ and hence controlling $\gap_k$ is challenging. We next  discuss a special case where such an assumption is achievable.  

Assume that we can construct $\signal_k$ as a near-minimizer,  
\begin{linenomath*}
\begin{equation}\label{eq:near}
\tik_k(\signal_k) \leq \inf_z \tik_k(z)  + \al_k \eta_k \,.
\end{equation}
\end{linenomath*}
Since $(\eta_k)_k$ is bounded,  \cite{obmann2022convergence} shows that  $(\signal_k)_k$ has a weakly convergent subsequence. After restriction to such a convergent subsequence and denoting its limit by $\signal^\ddag$ we get    $\varepsilon_k(\signal^\ddag)
    \leq \al_k \eta_k $.
With  $\al_k \asymp \delta_k$, according to Theorem~\ref{thm:ratesgap} we get the rates $\norm{\Ao \signal_k - \data_k}^2 = \mathcal{O}(\delta_k^2 + \delta_k \eta_k)$ and $\abs{\bregman_\xi(\signal_k, \signal^\ddag)} = \mathcal{O}(\delta_k + \eta_k)$.

\subsection{Iterative minimization}

In the following we assume that near minimization \eqref{eq:near} of $\tik_k$ is realized with some iterative algorithm.

\begin{corollary}[Iterative minimization] \label{cor:iterative}
Let $\mathcal{A}_k$ be an iterative algorithm for minimizing $\tik_k$ such that for the $n$-th iterate $\signal_{k,n}$  we have  $\tik_k(\signal_{k,n}) \leq \inf  \tik_k + f_k(n)$  with $f_k(n) \to 0$ as $n \to \infty$.  Let $\signal_{k, n(k)}$ satisfy  \eqref{eq:near}. Then
\begin{enumerate}[label =(\arabic*), topsep=0.5em, itemsep=0em]
    \item $\norm{\Ao \signal_{k,n(k)} - \data_k} = \mathcal{O}\Bigl(\sqrt{\delta_k^2 + \delta_k \eta_k}\Bigr)$
    \item $\bregman_G(\signal_{k,n(k)}, \signal^\ddag) = \mathcal{O}(\delta_k + \eta_k)$.
\end{enumerate}
If $\eta_k \to 0$ then $(\signal_{k, n(k)})_k$ converges to an $\reg$-minimizing solution of $\Ao \signal = \data$ in the Bregman-distance.
\end{corollary}

\begin{proof}
The rates follow from Theorem~\ref{thm:ratesgap} and because $\signal_{k, n(k)}$ is an $\al_k \eta_k$-minimizer. That $\signal^\ddag$ is an $\reg$-minimizing solution of $\Ao \signal = \data$ is shown similar to \cite{scherzer2009variational}.
\end{proof}

An important feature of  Corollary~\ref{cor:iterative} is that the points $\signal_{k, n(k)}$ can be obtained in a finite number of steps of the algorithm $\mathcal{A}_k$ and still result in convergence rates of order  $\mathcal{O}(\delta_k )$. Opposed to this, classical theory needs access to global minimizers which usually requires an infinite number of steps.

\begin{remark}[Convex regularizers]
The assumption of having access to an algorithm $\mathcal{A}_k$ is applicable if $\reg$ is  convex. In this case $\tik_k$  is convex and algorithms such as subgradient descent, heavy ball methods or accelerated gradient methods guarantee convergence  \cite{ghadimi2015global, nesterov1983method}. 
For such  algorithms  the number  of iterations in dependence of $\delta_k$ can be stated. Assume  $\eta_k = \delta_k$ and that the algorithm $\mathcal{A}_k$ is  convergnt with a rate of $f(n) = C n^{-\beta}$ for some $\beta > 0$. According to Corollary~\ref{cor:iterative}, we need to perform on the order of $(\alpha_k \delta_k)^{-1/\beta}$ to guarantee convergence to an $\reg$-minimizing solution. With the choice $\alpha_k \asymp \delta_k$ this yields that we need to perform on the order of $\delta_k^{-2 / \beta}$ iterations. For example, if $\beta = 1$, e.g. in the case of the heavy ball method \cite{ghadimi2015global}, we need on the order of $\delta_k^{-2}$ number of iterations which is comparable to the number of iterations for the Landweber iteration  \cite{EngHanNeu96}.
\end{remark}

Theorem~\ref{thm:ratesgap}  suggests that slower rates might occur if $\tik_k$ is not sufficiently small.    
Indeed, it is easy to construct examples where this is the case and convergence in the Bregman-distance is considerably slower than the rate $\mathcal{O}(\delta)$; compare  Example~\ref{ex:rates}.
As a consequence, this means that depending on the algorithm $\mathcal{A}_k$, its initialization and the choice of hyper-parameters that in some cases the required  number of iterations necessary is strict.

\begin{remark}[Comparison of inexactness results] \label{ref:comparison}
Finally, we briefly compare the results of Theorem~\ref{thm:ratesinexact} and Theorem~\ref{thm:ratesgap}. Both Theorems deal with the case of inexactness in the construction of the regularized solution, but they rely on different measures for  inexactness. As such these Theorems might be more suitable in  different situations. Assume  that in both cases  regularized solutions are constructed using an iterative algorithm.   The main advantage of Theorem~\ref{thm:ratesinexact} is that the condition $\norm{z_k} \leq \al_k \eta_k$ for some user-defined $\eta_k > 0$ is easily checkable  in an online manner during the the algorithm itself and no prior  information other than the tolerance $\eta_k$ is necessary. However, it may not be known a-priori how long this might take and the number of iterations could significantly increase depending on $\eta_k$. Opposed to this, Corollary~\ref{cor:iterative} gives an estimate of the number of iterations necessary. 
\end{remark}

\section{Numerical example} \label{sec:numerics}

We perform a simple test to numerically check the convergence rates derived in the previous sections in the context of iterative minimization. 

\subsection{Setting}

As inverse  problem  we adapt the ``Depth Profiling and Depth Resolution'' as presented in \cite[Section 7.8]{hansen2010discrete}. We take  $\X = \Y = L^2(0, \pi / 2)$ and the linear operator $\Ao \colon \X \to \Y$ defined by
\begin{linenomath*}
\begin{equation*}
    (\Ao \signal)(s) = \int_0^{\arcsin(\cos(s))} \exp(-\sin(\tau)) \cos(\tau) \signal(\tau) \mathrm{d}\tau,
\end{equation*}
\end{linenomath*}
for $\signal \in \X$ and $s \in (0, \pi / 2)$.

We consider the quadratic regularizer  $\reg = 1/2 \norm{\cdot}^2$ in which case according to Corollary~\ref{cor:convex}, Condition~\ref{cond:rate} holds true whenever the source condition \ref{cond:rate1} is satisfied. 
We construct near minimizers using gradient descent. As theoretical framework for the convergence rates we use Theorem~\ref{thm:ratesinexact} for which conditions on gradients can be checked during iteration; compare Remark~\ref{ref:comparison}.  The source condition is satisfied, whenever $\signal^\ddag \in \ran(\Ao^*)$. 

\subsection{Implementation details}

For the presented results we choose the true signal $\signal^\ddag = \Ao^* w$ with $w(s) = \cos(10 s) + \sin(5 s^2)$ which satisfies the source  condition by definition. 

We simulate noisy data $\datadelta$ by adding white noise for different noise levels $\delta \in \set{10^{-k} \colon k = 2, \dots, 7}$ to $\data = \Ao \signal^\ddag$.  We  choose the regularization parameter $\al_k = \delta_k$ and consider the Tikhonov functional $\tik_k =  \norm{\Ao (\cdot) - \data_k}^2/2 + \al_k  \norm{\cdot}^2/2$. We choose the tolerance level $\eta_k = \al_k^\beta$ with $\beta = 0.1$. Hence we stop the gradient descent iteration once we have $\norm{\nabla \tik_k(\signal_k)} \leq \delta_k^{1 + \beta}$. It should be noted, that since we stop gradient descent before convergence, the resulting $\signal_k$ depends on the initial value $\signal_0$ and hence we test for different choices of $\signal_0$ namely the constant $0$ and constant $1$ functions. The code for the numerical simulations is publicly available at \url{https://git.uibk.ac.at/c7021101/cpr-rates}.

\subsection{Results}

Numerical results are shown in Table~\ref{tab:rates} where the difference $\norm{\signal_k - \signal^\ddag}^2$ is given in dependence of $\delta$ for the two different initial values $\signal_0$. One notices that the convergence rate obtained with $\signal_0 = 0$ is way better than the one given in Theorem~\ref{thm:ratesinexact} and is closer to $\delta$ than $\delta^\beta$. The estimated rate is around $\beta_{\rm est}=0.99$. On the other hand, the convergence rate for the initial value $\signal_0 = 1$ is closer to $\delta^\beta$. The estimated rate in this case is $\beta_{\rm est}=0.22$. This shows that depending on the input-parameters of the algorithm the convergence rate obtained can significantly differ. This also indicates that without further assumptions better rates than the one in Theorem~\ref{thm:ratesinexact} can be expected.

\begin{table}
    \centering
    \begin{tabular}{c|c|c|c}
         \toprule
         $\delta$ & $\delta^\beta$ & $\signal_0 = 1$ & $\signal_0 = 0$ \\
         \midrule
         $10^{-2}$ & $6 \cdot 10^{-1}$ & $2 \cdot 10^{-1}$ & $1 \cdot 10^{-2}$ \\
         $10^{-4}$ & $4 \cdot 10^{-1}$ & $6 \cdot 10^{-2}$ & $2 \cdot 10^{-4}$ \\
         $10^{-6}$ & $3 \cdot 10^{-1}$ & $2 \cdot 10^{-2}$ & $3 \cdot 10^{-6}$ \\
         \bottomrule
    \end{tabular}
    \caption{Error $\norm{\signal_k - \signal^\ddag}^2$ in for different noise levels using gradient descent for two different initial values.}
    \label{tab:rates}
\end{table}

\section{Conclusion} \label{sec:conclusion}

In this paper, we have presented convergence rates in the absolute symmetric Bregman distance for the regularization of critical points under a classical source condition and an assumption on the nonconvexity of the regularizer $\reg$. This result has been generalized to inexact critical points, where the inexactness is measured in the magnitude of the gradient of the Tikhonov functional. Making the additional assumption that almost-minimizers can be achieved, we derived convergence rates in the absolute Bregman distance. A direct consequence is that, in contrast to the classical theory, access to global minimizers is not necessary for regularization, while known rates of $\mathcal{O}(\delta)$ are preserved in the absolute Bregman distance. We have also shown that near-minimizers on the order of $\delta^{-2 / \beta}$ iterations can be achieved using an iterative algorithm with rate $n^{-\beta}$.

We finally presented numerical simulations showing that non-exactness of the critical points can indeed lead to different convergence rates depending on the input parameters of the algorithm. Corollary~\ref{prop:necessity} suggests Morozov's discrepancy principle for choosing the regularization parameters, and establishing conditions for when this leads to a convergent regularization method is an interesting line of future research. Other directions of future work could focus more on the practical aspect of minimization and to derive conditions under which rates can be improved under an inexactness assumption.

%\bibliographystyle{abbrv}
%\bibliography{refs}

\end{document}